\documentclass[12pt,reqno]{amsart}

\usepackage{amsmath,amsthm,amssymb}
\usepackage{xspace,xcolor}
\usepackage[breaklinks,colorlinks,citecolor=teal,linkcolor=teal,urlcolor=teal,pagebackref,hyperindex]{hyperref}
\usepackage[alphabetic]{amsrefs}
\usepackage[all]{xy}

\renewcommand{\phi}{\varphi}
\renewcommand{\epsilon}{\varepsilon}
\renewcommand{\theta}{\vartheta}

\def\ZZ{{\mathbf Z}}

\def\AAA{{\mathbf A}}
\def\RR{{\mathbf R}}
\def\QQ{{\mathbf Q}}

\def\cO{\mathcal{O}}

\def\fra{\mathfrak{a}}
\def\frb{\mathfrak{b}}

\def\frm{\mathfrak{m}}

\DeclareMathOperator{\Spec}{Spec}

\DeclareMathOperator{\lct}{lct}

\DeclareMathOperator{\ord}{ord}
\DeclareMathOperator{\mult}{mult}

\DeclareMathOperator{\cosupp}{Cosupp}

\DeclareMathOperator{\mld}{mld}

\newcommand{\llbracket}{[\negthinspace[}
\newcommand{\rrbracket}{]\negthinspace]}

\newtheorem{lemma}{Lemma}[section]
\newtheorem{theorem}[lemma]{Theorem}

\newtheorem{proposition}[lemma]{Proposition}

\newtheorem{conjecture}[lemma]{Conjecture}

\theoremstyle{definition}

\newtheorem{remark}[lemma]{Remark}

\theoremstyle{remark}
\newtheorem*{remark*}{Remark}
\newtheorem*{note*}{Note}

\numberwithin{equation}{section}

\begin{document}

\title{A boundedness conjecture for minimal log discrepancies on a fixed germ}

\author[M.~Musta\c{t}\u{a}]{Mircea~Musta\c{t}\u{a}}
\address{Department of Mathematics, University of Michigan, 530 Church Street,
Ann Arbor, MI 48109, USA}
\email{mmustata@umich.edu}

\author[Y.~Nakamura]{Yusuke~Nakamura}
\address{Graduate School of Mathematical Sciences, University of Tokyo, 3-8-1 Komaba,
Meguro-Ku, Tokyo, 153-8914, Japan}
\email{nakamura@ms.u-tokyo.ac.jp}

\dedicatory{Dedicated to Lawrence Ein, on the~occasion of
his~sixtieth~birthday}

\thanks{The first author 
was partially supported by
 NSF grants DMS-1265256 and DMS-1401227. The second author was partially supported by the Grant-in-Aid for 
 Scientific Research (KAKENHI No. 25-3003) and the Program for Leading Graduate Schools, MEXT, Japan.}

\begin{abstract}
We consider the following conjecture: 
on a klt germ $(X,x)$, for every finite set $I$ there is a positive integer $\ell$ with the property that
for every $\RR$-ideal $\fra$ on $X$ with exponents in $I$,
there is a divisor $E$ over $X$ that computes the minimal log discrepancy $\mld_x(X,\fra)$ and such that its
discrepancy $k_E$ is bounded above by $\ell$. 
We show that this implies Shokurov's ACC conjecture for
minimal log discrepancies on a fixed klt germ and give some partial results towards the conjecture.
\end{abstract}

\maketitle

\section{Introduction}

One of the outstanding open problems in birational geometry is the Termination of Flips conjecture, which predicts that there are no infinite chains of
certain birational transformations (\emph{flips}). It is an insight due to Shokurov that this global problem can be reduced to conjectural properties of
invariants of singularities. A typical such property is the Ascending Chain Condition (ACC, for short) which predicts that in a fixed dimension, and with suitable restrictions on
the coefficients of the divisors involved, there are no infinite strictly increasing sequences of such invariants. There are two types of invariants that are important in this setting:
the log canonical thresholds and the minimal log discrepancies. As a rule, log canonical thresholds are easier to study and they are related to many other points of view on singularities.
In particular, Shokurov's ACC conjecture for log canonical thresholds has been proved (see \cite{dFEM2} for the smooth case, \cite{dFEM1} for the case of varieties with bounded singularities,
and \cite{HMX} for the general case). However, while the ACC property in this setting implies the termination of certain families of flips in an inductive setting (see \cite{Birkar} for the precise statement),
it does not allow proving any termination result in arbitrary dimension. It turns out that in order to do this one has to work with minimal log discrepancies (mlds, for short). In fact, Shokurov showed in
\cite{Shokurov} that two conjectural properties of mlds (the Semicontinuity conjecture and the ACC conjecture) imply termination of flips. The Semicontinuity conjecture is believed to be the easier of the two problems.
In fact, this is known in some cases (see \cite{EMY} for the case of smooth varieties and \cite{Nakamura} for the case of varieties with quotient singularities). 
In this paper we propose an approach towards 
Shokurov's ACC conjecture when we only consider mlds on a fixed germ of variety $(X,x)$. In particular, this would cover the case of smooth ambient varieties.

Before stating our main results, let us introduce some notation. We always assume that we work over an algebraically closed field of characteristic $0$. 
Let $X$ be a variety and $x\in X$ a (closed) point. We work with $\RR$-ideals $\fra$, that is, formal products $\fra=\prod_{j=1}^r\fra_j^{\lambda_j}$, where
the $\lambda_j$ are nonnegative real numbers and the $\fra_j$ are nonzero coherent ideals in $\cO_X$. We say that $\fra$ \emph{has exponents in} a set $I\subseteq\RR_{\geq 0}$ if $\lambda_j\in I$ for all $j$. We assume that $X$ is $\QQ$-Gorenstein
and denote by $\mld_x(X,\fra)$ the minimal log discrepancy of $(X,\fra)$ at $x$ (see \S 2 for the definition). This is a nonnegative real number if and only if $(X,\fra)$ is log canonical in some neighborhood of $x$;
otherwise, if $\dim(X)\geq 2$, then $\mld_x(X,\fra)=-\infty$.

In this paper we consider the following boundedness conjecture for mlds on a fixed germ.

\begin{conjecture}\label{conj_main}
Let $X$ be a klt variety and let $x\in X$. Given a finite subset $I\subset \RR_{\geq 0}$, there is a positive integer $\ell$ ${\rm (}$depending on $(X,x)$ and $I$${\rm )}$ such that
for every $\RR$-ideal $\fra$ on $X$ with exponents in $I$, there is a divisor $E$
that computes $\mld_x(X,\fra)$ and such that $k_E\leq \ell$. 
\end{conjecture}

We use the theory of generic limits of ideals
developed in \cite{dFM}, \cite{Kollar1}, and \cite{dFEM1} to show the weaker statement in which we bound the order along $E$ of the ideal defining
the point $x\in X$ (we expect this result to be useful for attacking the 
above conjecture). More precisely, we show the following:

\begin{theorem}\label{thm_bound_ord_point}
Let $X$ be a klt variety and $x\in X$ a point defined by the ideal $\frm_x$. For every finite subset $I\subset \RR_{\geq 0}$, there is a positive integer $\ell$ ${\rm (}$depending on $(X,x)$ and $I$${\rm )}$ such that the following conditions hold:
\begin{enumerate}
\item[i)] For every $\RR$-ideal $\fra$ with exponents in $I$ such that $\mld_x(\fra)>0$ and {\bf every} divisor $E$ over $X$ that computes $\mld_x(X,\fra)$, we have
$\ord_E(\frm_x)\leq\ell$.
\item[ii)] For every $\RR$-ideal $\fra$ with exponents in $I$ such that $\mld_x(\fra)\leq 0$, there is {\bf some} divisor $E$ over $X$ that computes $\mld_x(X,\fra)$ and such that
$\ord_E(\frm_x)\leq\ell$. 
\end{enumerate}
\end{theorem}

In a related direction, we also show that if $I$ is a finite set and $(X,x)$ is fixed, then there is a positive integer $\ell$ such that for every $\RR$-ideal $\fra$ on $X$ with exponents in $I$,
in order to check that $(X,\fra)$ is log canonical at $x$ it is enough to check that $a_E(X,\fra)\geq 0$ for all divisors $E$ with center $x$ and with $k_E\leq\ell$ (see 
Proposition~\ref{prop_LC}). This result admits a nice consequence concerning the characterization of log canonical pairs in terms of jet schemes
(see Proposition~\ref{consequence_jet_schemes}).

As farther evidence for the conjecture, we handle the two-dimensional case and the case of monomial ideals.

\begin{theorem}\label{dim2}
Conjecture~\ref{conj_main} holds if $\dim(X)=2$.
\end{theorem}

\begin{theorem}\label{monomial_case}
Conjecture~\ref{conj_main} holds if $(X,x)=(\AAA^n,0)$ and $\fra$ is a monomial $\RR$-ideal.
\end{theorem}

Our interest in the above conjecture is motivated by the following connection with 
Shokurov's ACC conjecture for minimal log discrepancies. Recall that a subset $I\subseteq\RR$ \emph{satisfies ACC} (\emph{DCC}) if 
it contains no infinite strictly increasing (resp., decreasing) sequences.

\begin{theorem}\label{thm_acc}
Let $X$ be a klt variety and $x\in X$ be a point such that the assertion in Conjecture~\ref{conj_main} holds for $(X,x)$ and for every finite subset $I\subset \RR_{\geq 0}$.
For every fixed DCC set
$J\subset\RR_{\geq 0}$, the set
$$\{\mld_x(X,\fra)\mid \fra\,\,\text{is}\,\,\text{an}\,\,\RR\text{-ideal on}\,\,X\,\,\text{with exponents in}\,\,J, \,(X,\fra)\,\,\text{is log canonical around}\,\,x\}$$
satisfies ACC.
\end{theorem}

We show that Conjecture~\ref{conj_main} is equivalent to two other conjectures on minimal log discrepancies. One of these is (a uniform version of) the Ideal-adic Semicontinuity conjecture for mlds
(see Conjecture~\ref{ideal_adic} for the precise formulation). This has been studied by Kawakita and various partial answers have been obtained in \cite{Kawakita4}, \cite{Kawakita3}, and \cite{Kawakita1}. The other conjecture is the Generic Limit conjecture on minimal log discrepancies, also studied by Kawakita in \cite{Kawakita2} (see Conjecture~\ref{conjecture:generic_limit}). 

\begin{theorem}\label{thm_equivalence}
Conjectures~\ref{conj_main}, \ref{conjecture:generic_limit}, and \ref{ideal_adic} are equivalent.
\end{theorem}

The paper is organized as follows. In \S 2 we recall the definition and some basic facts related to minimal log discrepancies. The following section 
is devoted to a review of generic limits and to the proof of Theorem~\ref{thm_bound_ord_point}. In \S 4 and \S 5 we prove Theorems~\ref{dim2} and \ref{monomial_case},
respectively. In \S 6  we prove
Theorems~\ref{thm_acc} and in \S 7 we prove Theorem~\ref{thm_equivalence}.

\subsection*{Acknowledgments}
We would like to thank Dale Cutkosky, Atsushi Ito, Mattias Jonsson, Masayuki Kawakita, Pierre Milman, and Michael Temkin for some useful discussions in connection with this work. 
We are especially indebted to Masayuki Kawakita for pointing out an error in an earlier version of this paper. 

It is a pleasure to dedicate this paper to Lawrence Ein, on the occasion of his sixtieth birthday. Lawrence's work has had a profound influence on the understanding of singularities of algebraic varieties and their
role in geometry. The first author, in particular, was introduced to this area through their conversations and collaboration. He would like to express his thanks and admiration.

\section{Minimal log discrepancies: definition and basic facts}

In this section we review the definition of minimal log discrepancies and set up the notation that we will use later in the paper.
For more details and for the proofs of some of the facts that we state, we refer to \cite{Ambro}.

We work over an algebraically closed  ground field, of characteristic $0$. Let $X$ be a variety (always assumed to be reduced and irreducible).
A \emph{divisor over} $X$ is a prime divisor $E$ on some normal variety $Y$, proper and birational over $X$. Such a divisor defines a discrete
valuation $\ord_E$ of the function field of $X$ and we identify two divisors if they give the same valuation. The image of $E$ on $X$ is the \emph{center} of $E$ on $X$
and it is denoted by $c_X(E)$. For a nonzero coherent ideal sheaf $\fra$ on $X$, one defines $\ord_E(\fra)$ as follows. If $E$ is a prime divisor on $Y$ and $t$ is a uniformizer of the DVR $\cO_{Y,E}$,
then we can write $\fra\cdot\cO_{Y,E}=(t^e)$ for some nonnegative integer $e$ and $\ord_E(\fra):=e$. Note that $\ord_E(\fra)>0$ if and only if $c_X(E)\subseteq\cosupp(\fra)$, where 
$\cosupp(\fra)$ is the support of $\cO_X/\fra$. 

Let $X$ be a normal variety. One says that $X$ is $\QQ$-Gorenstein if the canonical divisor $K_X$ is $\QQ$-Cartier. In this case, for every
proper, birational morphism $f\colon Y\to X$, with $Y$ normal, we consider the discrepancy divisor $K_{Y/X}$. If $E$ is a divisor over $X$ that appears
as a prime divisor on $Y$, then we denote by $k_E$ the coefficient of $E$ in $K_{Y/X}$ (this is independent of the choice of model $Y$). 

Recall that an $\RR$-ideal on $X$ is a formal product $\fra=\prod_{j=1}^r\fra_j^{\lambda_j}$, where each $\fra_j$ is a nonzero coherent ideal sheaf on $X$ and each $\lambda_j$ is a nonnegative real number. 
Given such $\fra$ and a divisor $E$ over $X$, we put
$$\ord_E(\fra):=\sum_{j=1}^r\lambda_j\cdot\ord_E(\fra_j).$$
If $\fra=\prod_{j=1}^r\fra_j^{\lambda_j}$ and $\frb=\prod_{i=1}^s\frb_i^{\mu_i}$ are two $\RR$-ideals and $\delta$ is a positive real number,
then we define the ideals 
$$\fra\cdot\frb:=\prod_{j=1}^r\fra_j^{\lambda_j}\cdot\prod_{i=1}^s\frb_i^{\mu_i}$$
and
$$\fra^{\delta}:=\prod_{j=1}^r\fra_j^{\delta\lambda_j}.$$
It is clear that in this case, if $E$ is a divisor over $X$, then $\ord_E(\fra\cdot \frb)=\ord_E(\fra)+\ord_E(\frb)$
and $\ord_E(\fra^{\delta})=\delta\cdot\ord_E(\fra)$.

Suppose now that
$X$ is  normal and $\QQ$-Gorenstein and $\fra$ is an $\RR$-ideal on $X$. For every
divisor $E$ over $X$, the \emph{log discrepancy} of $E$ with respect to $(X,\fra)$ is
$$a_E(X,\fra):=k_E+1-\ord_E(\fra).$$
The pair $(X,\fra)$ is \emph{log canonical} (\emph{klt}) if and only if 
$a_E(X,\fra)\geq 0$ (respectively, $>0$) for every divisor $E$ over $X$. When $\fra=\cO_X$, one simply says that
$X$ is log canonical (respectively, klt).

Consider a pair $(X,\fra)$, with $X$ a normal, $\QQ$-Gorenstein variety and $\fra$ an $\RR$-ideal on $X$. 
For every (closed) point $x\in X$, the \emph{minimal log discrepancy} of $(X,\fra)$ is given by
$$\mld_x(X,\fra):=\inf\{a_E(X,\fra)\mid E\,\,\text{is a divisor over}\, X\,\text{with}\,c_X(E)=x\}.$$
It is a basic fact that $\mld_x(X,\fra)\geq 0$ if and only if $(X,\fra)$ is log canonical in a neighborhood of $x$.
Moreover, if $\mld_x(X,\fra)<0$ and $\dim(X)\geq 2$, then $\mld_x(X,\fra)=-\infty$. 
One can also show that if $\mld_x(X,\fra)\geq 0$, then the infimum in the definition is in fact a minimum. 
Under this assumption, we say that a divisor $E$ over $X$ \emph{computes} $\mld_x(X,\fra)$ if $c_X(E)=x$
and $a_E(X,\fra)=\mld_x(X,\fra)$. When $\mld_x(X,\fra)<0$, we will say that $E$ computes $\mld_x(X,\fra)$
if $c_X(E)=x$ and $a_E(X,\fra)<0$. 

Recall that if $\fra$ is a nonzero ideal on $X$, then a \textit{log resolution} of $(X,\fra)$ is a proper, birational
morphism $\pi\colon Y\to X$ such that $Y$ is a smooth variety, the exceptional locus ${\rm Exc}(\pi)$ is a divisor,
$\fra\cdot\cO_Y=\cO_Y(-F)$ for some effective divisor
$F$ on $Y$, and $F+{\rm Exc}(\pi)$ has simple normal crossings. 
Since we are in characteristic $0$, log resolutions exist by Hironaka's theorem.
It is a basic result that if $X$ is a normal, $\QQ$-Gorenstein
variety, $x\in X$, and $\fra=\prod_{j=1}^r\fra_j^{\lambda_j}$ is an $\RR$-ideal on $X$, then for every log resolution $\pi\colon Y\to X$
of $(X,\frm_x\cdot\prod_{j=1}^r\fra_j)$, there is a divisor $E$ on $Y$ which computes $\mld_x(X,\fra)$.

\begin{proposition}\label{prop1}
Let $X$ be a normal, $\QQ$-Gorenstein variety, $\fra$ an $\RR$-ideal on $X$, and $x\in X$ a point defined by $\frm_x$.
If $\mld_x(X,\fra)>0$, then there is $\delta>0$ such that we have $\mld_x(X,\fra\cdot\frm_x^{\delta})=0$.
\end{proposition}

\begin{proof}
Let $\pi\colon Y\to X$ be a log resolution of $(X,\frm_x\cdot\prod_{j=1}^r\fra_j)$, where $\fra=\prod_{j=1}^r\fra_j^{\lambda_j}$. 
We see that we may take 
$$\delta=\min\left\{\frac{a_E(X,\fra)}{\ord_E(\frm_x)}\ \bigg| \  E\,\,\text{divisor on}\,\,Y\,\,\text{with}\,\,c_X(E)=x\right\}.$$
\end{proof}

In what follows we will also make use of the notion of log canonical threshold. Suppose that $X$ is a log canonical variety and $x\in X$.
If $\fra$ is an $\RR$-ideal on $X$, then the \emph{log canonical threshold} of $(X,\fra)$ at $x$ is given by
$$\lct_x(X, \fra):=\inf\left\{\frac{k_E+1}{\ord_E(\fra)}\ \bigg| \  E\,\,\text{divisor over}\,\,X\,\,\text{with}\,\, x\in c_X(E)\right\}.$$
In fact, if $\fra=\prod_{j=1}^r\fra_j^{\lambda_j}$ and $\pi\colon Y\to X$ is a log resolution of $(X,\prod_{j=1}^r\fra_j)$, then
there is a divisor $E$ on $Y$ that computes $\lct_x(X,\fra)$, that is, $\lct_x(X,\fra)=(k_E+1)/\ord_E(\fra)$ and $x\in c_X(E)$. 
Note that we have $\mld_x(X,\fra)\geq 0$ if and only if $\lct_x(X,\fra)\geq 1$. 

We collect in the next proposition a few well-known properties of minimal log discrepancies and log canonical thresholds.
The proof is straightforward and we omit it.

\begin{proposition}\label{general_properties}
Let $X$ be a log canonical variety and let $x\in X$ be defined by $\frm_x$. If $\fra_1,\ldots,\fra_r,\frb_1,\ldots,\frb_r$ are nonzero ideals on $X$
and $\lambda_1,\ldots,\lambda_r,\mu_1,\ldots,\mu_r$ are nonnegative real numbers, then the following hold:
\begin{enumerate}
\item[i)] If $\fra_j\subseteq \frb_j$ for every $j$, then 
$$\mld_x(X,\fra_1^{\lambda_1}\dotsm\fra_r^{\lambda_r})\leq \mld_x(X,\frb_1^{\lambda_1}\dotsm\frb_r^{\lambda_r})\quad\text{and}\quad \lct_x(X,\fra_1^{\lambda_1}\dotsm\fra_r^{\lambda_r})\leq
\lct_x(X,\frb_1^{\lambda_1}\dotsm\frb_r^{\lambda_r}).$$
\item[ii)] If $\lambda_j\leq\mu_j$ for every $j$, then 
$$\mld_x(X,\fra_1^{\lambda_1}\dotsm\fra_r^{\lambda_r})\geq\mld_x(X,\fra_1^{\mu_1}\dotsm\fra_r^{\mu_r})
\quad\text{and}\quad \lct_x(X,\fra_1^{\lambda_1}\dotsm\fra_r^{\lambda_r})\geq \lct_x(X,\fra_1^{\mu_1}\dotsm\fra_r^{\mu_r}).$$
\item[iii)] For every $\delta>0$, we have
$$\lct_x(X,\fra_1^{\delta\lambda_1}\dotsm\fra_r^{\delta\lambda_r})=\delta^{-1}\cdot \lct_x(X,\fra_1^{\lambda_1}\dotsm\fra_r^{\lambda_r}).$$
\item[iv)] If $E$ is a divisor over $X$ with $c_X(E)=x$ and $E$ computes $\mld_x(X,\fra_1^{\lambda_1}\dotsm\fra_r^{\lambda_r})$
${\rm (}$resp., $\lct_x(X,\fra_1^{\lambda_1}\dotsm\fra_r^{\lambda_r})$${\rm )}$ and if $d$ is a positive integer such that
$d\cdot \ord_E(\frm_x)\geq\ord_E(\fra_j)$ for all $j$, then $E$  computes
$\mld_x\big(X,\prod_{j=1}^r(\fra_j+\frm_x^d)^{\lambda_j}\big)$ and this is equal to $\mld_x\big(X,\prod_{j=1}^r\fra_j^{\lambda_j}\big)$
${\rm (}$resp., $E$ computes $\lct_x\big(X,\prod_{j=1}^r(\fra_j+\frm_x^d)^{\lambda_j}\big)$ and this is equal to $\lct_x\big(X,\prod_{j=1}^r\fra_j^{\lambda_j}\big)$${\rm )}$.
\end{enumerate}
\end{proposition}

In the next section we will need to work in a more general setting than the one described above, in which $X$ is allowed to be a normal, excellent,
$\QQ$-Gorenstein scheme of characteristic $0$ (that is, all the residue fields of $X$ have characteristic $0$). All the above definitions extend to this setting.
For details, in particular for the precise definitions of $K_X$ and $K_{Y/X}$ in this framework, we refer to \cite[Appendix A]{dFEM1}.

\section{Generic limits: bounding the order of the ideal of the point} \label{section:generic_limit}

Our goal in this section is to prove Theorem~\ref{thm_bound_ord_point}.
The proof uses generic limits of sequences of ideals. Such a construction based on nonstandard methods was given in \cite{dFM} and a different one, 
with the same properties but based on
sequences of generic points was later given in \cite{Kollar1}. In what follows we simply recall the basic properties of such a construction, following \cite{dFEM1}.

Let $X$ be a klt variety over $k$ and $x\in X$ a closed point. Given a
positive integer $r$ and $r$ 
 sequences of coherent sheaves of ideals $(\fra^{(i)}_j)_{i\geq 1}$ on $X$ for $1\leq j\leq r$,
we get an affine klt scheme $\widetilde{X}$, a closed point $\widetilde{x}\in\widetilde{X}$, and $r$ ideals $\widetilde{\fra}_1,\ldots,\widetilde{\fra}_r$ on $\widetilde{X}$
(note that it can happen for some $\widetilde{\fra}_j$ to be zero).
In \cite{dFEM1} one allows the variety $X$ to vary as well; since we assume that this is not the case, it is easy
to describe $\widetilde{X}$. If some affine neighborhood of $x$ in $X$ is defined  in some $\AAA_k^N$ by $h_1,\ldots,h_s$, then
$\widetilde{X}=\Spec (K\llbracket x_1,\ldots,x_N\rrbracket/(h_1,\ldots,h_s))$ for some algebraically closed field extension $K$ of $k$,
and $\widetilde{x}$ is the unique closed point of $\widetilde{X}$.
If for some $j$ we have $\fra^{(i)}_j=\frm_x$ for all $i\gg 0$, then $\widetilde{\fra}_j$ is the ideal $\frm_{\widetilde{x}}$ defining $\widetilde{x}$.
We collect in the next proposition some basic properties of this construction. 

\begin{proposition}\label{properties_generic_limit}
With the above notation, the following hold:
\begin{enumerate}
\item[i)] If $\widetilde{\fra}_j=0$, then for every $q$, there are infinitely many $i$ such that $\fra_j^{(i)}\subseteq\frm_x^q$.
\item[ii)] For every $d$, there is an infinite subset $\Lambda=\Lambda_d\subset\ZZ_{>0}$ such that for every $i\in \Lambda$ and for every $\lambda_1,\ldots,\lambda_r\in\RR_{\geq 0}$, we have
$$\lct_{\widetilde{x}}\big(\widetilde{X},\prod_{j=1}^r(\widetilde{\fra}_j+\frm_{\widetilde{x}}^d)^{\lambda_j}\big)=\lct_x\big(X,\prod_{j=1}^r(\fra_j^{(i)}+\frm_x^d)^{\lambda_j}\big).$$
\item[iii)] For every $\lambda_1,\ldots,\lambda_r\in\RR_{>0}$, if we consider the $\RR$-ideals  $\fra^{(i)}=\prod_{j=1}^r(\fra^{(i)}_j)^{\lambda_j}$ and $\widetilde{\fra}=\prod_{j=1}^r\widetilde{\fra}_j^{\lambda_j}$,
then $\lct_{\widetilde{x}}(\widetilde{X},\widetilde{\fra})$ is a limit point of the set $\{\lct_x(X,\fra^{(i)})\mid i\geq 1\}$ ${\rm (}$with the convention that if some $\widetilde{\fra}_j=0$, then $\lct_{\widetilde{x}}(\widetilde{X},\widetilde{\fra})=0$${\rm )}$.
\item[iv)] Suppose that $\widetilde{\fra}_j\neq 0$ for all $j$.
If $E$ is a divisor over $\widetilde{X}$ with $c_{\widetilde{X}}(E)=\widetilde{x}$ and such that $E$ computes $\lct_{\widetilde{x}}(\widetilde{X},\widetilde{\fra})$, then for every $d\gg 0$ there is an
infinite subset $\Lambda'=\Lambda'_d(E,\lambda_1,\ldots,\lambda_r)\subset\ZZ_{>0}$ with the following property:  for every $i\in \Lambda'$ there is a divisor $E_i$ over $X$ that computes
$\lct_x\big(X,\prod_{j=1}^r(\fra^{(i)}_j+\frm_x^d)^{\lambda_j}\big)$, which is equal to $\lct_{\widetilde{x}}\big(\widetilde{X},\prod_{j=1}^r(\widetilde{\fra}_j+\frm_{\widetilde{x}}^d)^{\lambda_j}\big)$
and we have $\ord_E(\frm_{\widetilde{x}})=\ord_{E_i}(\frm_x)$ ${\rm (}$in particular, we have $c_X(E_i)=x$${\rm )}$, $k_{E_i}=k_E$, and $\ord_E(\widetilde{\fra}_j+\frm_{\widetilde{x}}^d)=\ord_{E_i}(\fra^{(i)}_j+\frm_x^d)$ for $1\leq j\leq r$. 
\end{enumerate}
\end{proposition}

\begin{proof}
For the assertion in i), see \cite[Lemma~3.1]{dFEM1}. The statements in ii), iii), and iv) follow from \cite[Proposition~3.3 and Corollary~3.4]{dFEM1}. The only assertion that is not explicitly
mentioned in \emph{loc. cit.} is the one in iv) saying that $k_E=k_{E_i}$. However, 
by taking $d$ such that $d\geq \ord_E(\widetilde{\fra}_j)$ for every $j$, we may assume that with $\widetilde{\frb}=\prod_{j=1}^r(\widetilde{\fra}_j+\frm_{\widetilde{x}}^d)^{\lambda_j}$, we have $\ord_E(\widetilde{\fra})=\ord_E(\widetilde{\frb})$ and 
 $\lct_{\widetilde{x}}(\widetilde{X},\widetilde{\frb})=\lct_{\widetilde{x}}(\widetilde{X},\widetilde{\fra})$ (see Proposition~\ref{general_properties}). 
 We now conclude that $k_E=k_{E_i}$ from the other assertions.
\end{proof}

\begin{remark}\label{remark_properties_generic_limit}
With the notation in the above proposition, we also have the following variant of the assertion in Proposition~\ref{properties_generic_limit}: 
for every $d$, there is an infinite subset $\Lambda=\Lambda_d\subset\ZZ_{>0}$ such that for every $i\in \Lambda$ and for every $\lambda_1,\ldots,\lambda_r\in\RR_{\geq 0}$, we have
$$\mld_{\widetilde{x}}\big(\widetilde{X},\prod_{j=1}^r(\widetilde{\fra}_j+\frm_{\widetilde{x}}^d)^{\lambda_j}\big)=\mld_x\big(X,\prod_{j=1}^r(\fra_j^{(i)}+\frm_x^d)^{\lambda_j}\big).$$
The proof is the same as in the case of log canonical thresholds (see \cite[Proposition~3.3]{dFEM1}), the key point being that minimal log discrepancies are constant generically in a family.
More precisely, suppose that $x\in X$ is fixed, $T$ is an arbitrary variety, and $\frb_1,\ldots,\frb_r$ are ideals on $X\times T$ such that each $\frb_{j,t}=\frb_j\cdot\cO_{X\times\{t\}}$, with $1\leq j\leq r$ and
$t\in T$, is nonzero. In this case, there is an open subset $U$ of $T$ such that for each $\lambda_1,\ldots,\lambda_r\in\RR_{\geq 0}$, the minimal log discrepancy
$$\mld_x\big(X,\prod_{j=1}^r\frb_{j,t}^{\lambda_j}\big)$$
is constant for $t\in U$.
Moreover, the set $\Lambda$ can be chosen such that the ideals 
$\widetilde{\fra}_1,\ldots,\widetilde{\fra}_r$ are again generic limits of the sequences
$({\fra}^{(i)}_1)_{i\in\Lambda},\ldots, ({\fra}^{(i)}_r)_{i\in\Lambda}$.
\end{remark}

We can now prove the main result of this section.

\begin{proof}[Proof of Theorem~\ref{thm_bound_ord_point}]
We argue by contradiction. If the conclusion of the theorem fails, then we can find a sequence of $\RR$-ideals $(\fra^{(i)})_{i\geq 1}$ with
exponents in $I$ such that one of the following things happens:

\noindent {\bf Case 1}. We have $\mld_x(X,\fra^{(i)})>0$ for all $i$ and for every $i$ there is a divisor $E_i$ over $X$ that
computes $\mld_x(X,\fra^{(i)})$ and such that $\lim_{i\to\infty}\ord_{E_i}(\frm_x)=\infty$.

\noindent {\bf Case 2}. We have $\mld_x(X,\fra^{(i)})=0$ for all $i$ and for every choice of divisors $E_i$ over $X$ such that 
$E_i$ computes $\mld_x(X,\fra^{(i)})$, we have $\lim_{i\to\infty}\ord_{E_i}(\frm_x)=\infty$.

\noindent {\bf Case 3}. We have $\mld_x(X,\fra^{(i)})<0$ for all $i$ and for every choice of divisors $E_i$ over $X$ such that
$E_i$ computes $\mld_x(X,\fra^{(i)})$, we have $\lim_{i\to\infty}\ord_{E_i}(\frm_x)=\infty$.

Suppose that $\lambda_1,\ldots,\lambda_r$ are the nonzero elements of $I$. We may assume that for every $i$ we can write
$\fra^{(i)}=\prod_{j=1}^r(\fra^{(i)}_j)^{\lambda_j}$. 
We use the generic limit construction to construct $\widetilde{x}\in \widetilde{X}$ and an ideal $\widetilde{\fra}_j$ on $\widetilde{X}$ corresponding to the
sequence $(\fra^{(i)}_j)_{i\geq 1}$ for $1\leq j\leq r$. Let $\widetilde{\fra}$ be the $\RR$-ideal on $\widetilde{X}$ given by
$\widetilde{\fra}=\prod_{j=1}^r\widetilde{\fra}_j^{\lambda_j}$. When some $\widetilde{\fra}_j$ is zero, we make the convention that $\widetilde{\fra}=0$ and
$\lct_{\widetilde{x}}(\widetilde{X}, \widetilde{\fra})=0$.

Suppose first that we are either in Case 1 or in Case 2. Note that since $\lct_x(X,\fra^{(i)})\geq 1$ for every $i$, it 
follows from Proposition~\ref{properties_generic_limit} that $\lct_{\widetilde{x}}(\widetilde{X},\widetilde{\fra})\geq 1$. In particular, each $\widetilde{\fra}_j$ is nonzero and we have $\mld_{\widetilde{x}}(\widetilde{X},\widetilde{\fra})\geq 0$.
Let us consider first the case when $\mld_{\widetilde{x}}(\widetilde{X},\widetilde{\fra})>0$. It follows from Proposition~\ref{prop1} that there is $\delta>0$ such that $\lct_{\widetilde{x}}(\widetilde{X},\widetilde{\fra}\cdot\frm_{\widetilde{x}}^{\delta})=1$.
In this case there are infinitely many $i$ such that
$\lct_x(X,\fra^{(i)}\cdot\frm_x^{\delta})\geq 1$. Indeed, if this is not the case, then $\lct_x(X,\fra^{(i)}\cdot\frm_x^{\delta})<1$ for all $i\gg 0$. On the other hand, it follows from Proposition~\ref{properties_generic_limit}
that $\lct_{\widetilde{x}}(\widetilde{X},\widetilde{\fra}\cdot\frm_{\widetilde{x}}^{\delta})=1$ is a limit point of the set $\{\lct_x(X,\fra^{(i)}\cdot\frm_x^{\delta})\mid i\geq 1\}$. This contradicts the fact 
that the set $\{\lct_x(X,\fra^{(i)}\cdot\frm_x^{\delta})\mid i\geq 1\}$ satisfies ACC (see \cite[Theorem~4.2]{dFEM1}).

For every $i$ such that $\lct_x(X,\fra^{(i)}\cdot\frm_x^{\delta})\geq 1$ and for every divisor $E_i$ that computes $\mld_x(X,\fra^{(i)})$, we obtain
$$\mld_x(X,\fra^{(i)})=k_{E_i}+1- \ord_{E_i}(\fra^{(i)})\geq \delta\cdot\ord_{E_i}(\frm_x).$$ 
Therefore
$$\ord_{E_i}(\frm_x)\leq \frac{\mld_x(X,\fra^{(i)})}{\delta}\leq \frac{\mld_x(X)}{\delta}$$
for infinitely many $i$, contradicting the fact that, by assumption, we can choose such divisors $E_i$ with $\lim_{i\to\infty}\ord_{E_i}(\frm_x)=\infty$.

We now consider the case when $\mld_{\widetilde{x}}(\widetilde{X},\widetilde{\fra})=0$ (still assuming that we are either in Case 1 or in Case 2). If 
 $F$ is a divisor over $\widetilde{X}$ that computes $\mld_{\widetilde{x}}(\widetilde{X},\widetilde{\fra})$, then it follows from
Proposition~\ref{properties_generic_limit} that for $d\gg 0$, we can find an infinite subset
$\Gamma'=\Gamma'_d(F,\lambda_1,\ldots,\lambda_r)\subset\ZZ_{>0}$ such that the following holds. For every
$i\in \Gamma'$ we have a divisor
$F_i$ over $X$ with $k_F=k_{F_i}$, $\ord_F(\frm_{\widetilde{x}})=\ord_{F_i}(\frm_x)$ (in particular,
$c_X(F_i)=x$), and such that if we put $\frb^{(i)}=\prod_{j=1}^r(\fra^{(i)}_j+\frm_x^d)^{\lambda_j}$ and $\widetilde{\frb}=\prod_{j=1}^r(\widetilde{\fra}_j+\frm_{\widetilde{x}}^d)^{\lambda_j}$, then
 $\ord_F(\widetilde{\frb})=\ord_{F_i}(\frb^{(i)})$.
By taking $d\geq \ord_F(\widetilde{\fra}_j)$ for every $j$, we may assume that $\ord_F(\widetilde{\fra})=\ord_F(\widetilde{\frb})$.
We conclude that
 $$0=a_F(\widetilde{X},\widetilde{\fra})=k_F+1-\ord_F(\widetilde{\fra})=k_{F_i}+1-\ord_{F_i}(\frb^{(i)})=a_{F_i}(X,\frb^{(i)})\geq a_{F_i}(X,\fra^{(i)})\geq 0$$
 for every $i\in\Gamma'$.
 In Case 1, this already gives a contradiction, since the last inequality is strict. If we are in Case 2, we conclude that the divisor $F_i$ computes $\mld_x(X,\fra^{(i)})$.
 By assumption, we must have $\ord_{F_i}(\frm_x)\to\infty$, contradicting the fact that $\ord_{F_i}(\frm_x)$ is constant for $i\in \Gamma'$.
 
 Finally, suppose that we are in Case 3.  Let us assume first that every $\widetilde{\fra}_j$ is nonzero. Since $\lct_x(X,\fra^{(i)})<1$ for every $i$ and the set 
 $\{\lct_x(X,\fra^{(i)})\mid i\geq 1\}$ has
 $\lct_{\widetilde{x}}(\widetilde{X},\widetilde{\fra})$ as a limit point  by Proposition~\ref{properties_generic_limit}, it follows that
 $\lct_{\widetilde{x}}(\widetilde{X},\widetilde{\fra})<1$ (recall that the set
 $\{\lct_x(X,\fra^{(i)})\mid i\geq 1\}$ satisfies ACC by  \cite[Theorem~4.2]{dFEM1}). Therefore $\mld_{\widetilde{x}}(\widetilde{X},\widetilde{\fra})<0$ and consider a divisor $G$
 over $\widetilde{X}$, with $c_{\widetilde{X}}(G)=\widetilde{x}$ and with $a_G(\widetilde{X},\widetilde{\fra})<0$. We now argue as above: we can find an infinite set $\Gamma''\subset\ZZ_{>0}$
 such that the following holds. For every $i\in \Gamma''$ we have a divisor $G_i$ over $X$ with $k_{G}=k_{G_i}$, $\ord_G(\frm_{\widetilde{x}})=\ord_{G_i}(\frm_x)$ (in particular,
$c_X(G_i)=x$), and such that 
 $\ord_G(\widetilde{\frb})=\ord_{G_i}(\frb^{(i)})$, where $\frb$ and $\frb^{(i)}$ are defined as above. Furthermore, we may assume 
that $\ord_G(\widetilde{\fra})=\ord_G(\widetilde{\frb})$ and we conclude that
 $$0>a_G(\widetilde{X},\widetilde{\fra})=k_G+1-\ord_G(\widetilde{\fra})=k_{G_i}+1-\ord_{G_i}(\frb^{(i)})=a_{G_i}(X,\frb^{(i)})\geq a_{G_i}(X,\fra^{(i)})$$
 for every $i\in \Gamma''$. Since $\ord_{G_i}(\frm_x)$ is constant for all $i\in \Gamma''$, this gives a contradiction. 
 
 Let us consider now the case when some $\widetilde{\fra}_j$ is zero. Let $T$ be a fixed divisor over $X$ with $c_X(T)=x$ and let $q$ be a positive integer
 with $q>\frac{k_T+1}{\lambda_j\cdot\ord_T(\frm_x)}$. Since $\widetilde{\fra}_j$ is zero, it follows from Proposition~\ref{properties_generic_limit} that
there are infinitely many $i$ with $\fra_j^{(i)}\subseteq\frm_x^q$. In this case we have
$$a_T(X,\fra^{(i)})\leq a_T(X,\frm_x^{\lambda_j q})=k_T+1-\lambda_jq\cdot\ord_T(\frm_x)<0.$$
Therefore $T$ computes $\mld_x(X,\fra^{(i)})$ for infinitely many $i$, a contradiction. This completes the proof of the theorem.
\end{proof}

While by using generic limits we cannot get a proof for the full statement in Conjecture~\ref{conj_main}, we also obtain the
following related statement.

\begin{proposition}\label{prop_LC}
Let $X$ be a klt variety and $x\in X$ a closed point. If $I\subset\RR_{\geq 0}$ is a finite set, then there is a positive integer $\ell$ such that
for every $\RR$-ideal with exponents in $I$, if $a_E(X,\fra)\geq 0$ for all divisors $E$ over $X$ with $c_X(E)=x$ and $k_E\leq\ell$, then
$(X,\fra)$ is log canonical at $x$.
\end{proposition}

\begin{proof}
Suppose that the conclusion of the proposition fails. In this case we can find a sequence of $\RR$-ideals $\fra^{(i)}$ on $X$, with exponents in $I$,
such that each $(X,\fra^{(i)})$ is not log canonical at $x$, but 
$a_E(X,\fra^{(i)})\geq 0$ for all divisors $E$ over $X$ with $c_X(E)=x$ and $k_E\leq i$. 
Let $\lambda_1,\ldots,\lambda_r$ be the nonzero elements in $I$ and let us write
$$\fra^{(i)}=\prod_{j=1}^r(\fra^{(i)}_j)^{\lambda_j}.$$
We use the generic limit construction to produce $\widetilde{x}\in \widetilde{X}$ and ideals $\widetilde{\fra}_j$ on $\widetilde{X}$ corresponding to the
sequences $(\fra^{(i)}_j)_{i\geq 1}$ for $1\leq j\leq r$. Let $\widetilde{\fra}$ be the $\RR$-ideal on $\widetilde{X}$ given by
$$\widetilde{\fra}=\prod_{j=1}^r\widetilde{\fra}_j^{\lambda_j}.$$ 
When some $\widetilde{\fra}_j$ is zero, we make the convention that $\widetilde{\fra}=0$. 

Our assumption implies $\lct_x(X,\fra^{(i)})<1$ for every $i$. Recall that $\lct_{\widetilde{x}}(\widetilde{X},\widetilde{\fra})$ is a limit point of the sequence
$\big(\lct_x(X,\fra^{(i)})\big)_{i\geq 1}$ by Proposition~\ref{properties_generic_limit} iii). On the other hand, this sequence contains no strictly increasing subsequences by
\cite[Theorem~4.2]{dFEM1}. Therefore $\lct_{\widetilde{x}}(\widetilde{X},\widetilde{\fra})<1$ and the pair $(\widetilde{X},\widetilde{\fra})$ is not log canonical at $\widetilde{x}$. Let $E$ be a divisor over $\widetilde{X}$ with
center $\widetilde{x}$ and such that $a_E(\widetilde{X},\widetilde{\fra})<0$. If $d\in\ZZ_{>0}$ is large enough, but fixed, then we clearly have
$$a_E\big(\widetilde{X}, \prod_j(\widetilde{\fra}_j+\frm_{\widetilde{x}}^{d})^{\lambda_j}\big)=a_E(\widetilde{X},\widetilde{\fra})<0.$$
On the other hand, it follows from Proposition~\ref{properties_generic_limit} iv) that there are infinitely many $i$ for which we can find divisors $E_i$ over $X$ with center $x$,
such that $k_{E_i}=k_E$ and
$$a_{E_i}\big(X,\prod_j(\fra^{(i)}_j+\frm_x^{d})^{\lambda_j}\big)=a_{E}\big(\widetilde{X}, \prod_j(\widetilde{\fra}_j+\frm_{\widetilde{x}}^{d})^{\lambda_j}\big)<0.$$
Since 
$$a_{E_i}\big(X,\prod_j(\fra^{(i)}_j)^{\lambda_j}\big)\leq a_{E_i}\big(X,\prod_j(\fra^{(i)}_j+\frm_x^d)^{\lambda_j}\big)<0$$
and $k_{E_i}=k_E$ for infinitely many $i$, we contradict our assumption. This completes the proof of the proposition.
\end{proof}

The assertion in Proposition~\ref{prop_LC} has an interesting consequence in connection with the description of log canonical pairs
in terms of jet schemes, when the ambient variety is smooth. This will not play any role in the following sections, so the reader not interested in jet schemes
could skip this part.

Recall that if $X$ is a smooth variety, $Y$ is a closed subscheme of $X$
defined by the nonzero ideal $\fra$, and $q\in\RR_{\geq 0}$, then the pair $(X,\fra^q)$ is log canonical if and only if
$$\dim(Y_m)\leq (m+1)(\dim(X)-q)\quad\text{for all}\quad m\geq 0,$$
where $Y_m$ is the $m^{\rm th}$ jet scheme of $Y$ (see \cite[Corollary~3.2]{ELM}).
For the definition and basic properties of jet schemes and contact loci, we refer to \cite{ELM}. 
Proposition~\ref{prop_LC} implies that if the dimension of $X$ and $q\in\RR_{\geq 0}$ are fixed, then
it is enough to check the dimensions of only a prescribed number of jet schemes.

\begin{proposition}\label{consequence_jet_schemes}
Given $n\geq 1$ and $q\in\RR_{\geq 0}$, there is a positive integer $N$ that satisfies the following property.
For every smooth $n$-dimensional variety $X$ and for every closed subscheme $Y$ of $X$ defined by 
a nonzero ideal $\fra$, the pair $(X,\fra^q)$ is log canonical if and only if 
$$\dim(Y_m)\leq (m+1)(n-q)\quad\text{for all}\quad m\leq N.$$
\end{proposition}

\begin{proof}
The case $q=0$ is trivial (the pair is always log canonical in this case, hence any $N$ will work), hence we assume from now on $q>0$.
We first consider the case when $X=\AAA^n$ and choose $\ell$ given by Proposition~\ref{prop_LC},
such that for every nonzero ideal $\fra$ in $\AAA^n$, if $a_E(\AAA^n,\fra^q)\geq 0$ for all divisors $E$ over $\AAA^n$ with center at the origin and $k_E\leq\ell$, then
$(\AAA^n,\fra^q)$ is log canonical at $0$. Let $N=\lfloor \frac{\ell+1}{q}\rfloor$, where $\lfloor u\rfloor$ denotes the largest integer $\leq u$. 
We show that if $\fra$ is a nonzero ideal defining the subscheme $Y$ of $\AAA^n$ such that
$\dim(Y_m)\leq (m+1)(n-q)$ for all $m\leq N$, then $(\AAA^n,\fra^q)$ is log canonical at $0$. 

Indeed, if $(\AAA^n,\fra)$ is not log canonical at $0$, then it follows by assumption that
there is a divisor $E$ over $\AAA^n$ with center $0$ such that $k_E\leq\ell$ and
$k_E+1<q\cdot \alpha_E$, where $\alpha_E=\ord_E(\fra)$. Since $\alpha_E$ is an integer, it follows that 
$\alpha_E\geq m+1$, where $m=\lfloor \frac{k_E+1}{q}\rfloor$.
Let $f\colon Y\to \AAA^n$ be a log resolution of $(\AAA^n,\fra)$ such that
$E$ appears as a divisor on $Y$. It follows from \cite[Theorem~2.1]{ELM} that if $C=\overline{f_{\infty}({\rm Cont}^{\geq 1}(E))}$,
then 
$$C\subseteq {\rm Cont}^{\geq \alpha_E}(\fra)\subseteq {\rm Cont}^{\geq (m+1)}(\fra)\quad\text{and}\quad {\rm codim}(C)=k_E+1.$$
We thus conclude that
$$\dim(Y_m)=(m+1)n-{\rm codim}({\rm Cont}^{\geq (m+1)}(\fra))\geq (m+1)n-{\rm codim}(C)>(m+1)(n-q).$$
Since $m\leq N$, this proves our assertion.

Suppose now that $X$ is an arbitrary smooth $n$-dimensional variety and $\fra$ is a nonzero ideal, defining the closed subscheme $Y$
of $X$, such that 
$$\dim(Y_m)\leq (m+1)(n-q)\quad\text{for all}\quad m\leq N.$$
We show that for every $x\in X$, the pair $(X,\fra^q)$ is log canonical at $x$.
Since $X$ is smooth, after possibly replacing $X$ by an open neighborhood of $X$, we may assume that
we have an \'{e}tale morphism $g\colon X\to\AAA^n$, with $g(x)=0$. Let $\frm_x$ denote the ideal defining $x$
and for every $d\geq 1$, let $\fra_d=\fra+\frm_x^d$, defining the subscheme $V(\fra_d)$ of $X$. For every such $d$, there is an ideal $\frb_d$ on $\AAA^n$ defining
a subscheme $V(\frb_d)$ supported at $0$ and such that $\frb_d\cdot\cO_X=\fra_d$. 
Note that for every $d$ and $m$, we have 
$$V(\frb_d)_m\simeq V(\fra_d)_m\hookrightarrow Y_m,$$
hence by assumption 
$$\dim\big(V(\frb_d)_m\big)\leq (m+1)(n-q)\quad\text{for all}\quad m\leq N.$$
As we have seen, this implies that 
$(\AAA^n,\frb_d^q)$ is log canonical.
Since $g$ is \'{e}tale, we have $\lct_x(X,\fra_d)=\lct_0(\AAA^n,\frb_d)\geq q$ for every $d$, while

$$\lct_x(X,\fra)=\lim_{d\to\infty}\lct_x(X,\fra_d)$$
(see, for example, \cite[Proposition~2.15]{dFEM1}).
We conclude that $\lct_x(X,\fra)\geq q$, that is, the pair $(X,\fra^q)$ is log canonical at $x$.
This completes the proof of the proposition.
\end{proof}

\section{A proof of the conjecture in dimension 2}

We begin with the following convexity property of log discrepancies from \cite[Proposition 2.37]{Kollar3}. 

\begin{proposition}\label{proposition:conv}
Let $X$ be a surface and $\fra$ an $\RR$-ideal on $X$ such that $(X,\fra)$ is log canonical, and $f \colon Y \to X$ a birational morphism 
from a smooth surface $Y$. Assume that $a_E (X, \fra) \le 1$ for every $f$-exceptional divisor $E$. 
If $E_1$, $E_2$, and $E_3$ are $f$-exceptional prime divisors that satisfy the following conditions: 
\begin{enumerate}
\item $E_1$ meets both $E_2$ and $E_3$, and
\item $E_1$ has the self-intersection number $E_1 ^2 \le -2$,
\end{enumerate}
then $a_1 \le \frac{1}{2}(a_2 + a_3)$, where $a_i = a_{E_i} (X, \fra)$. 
\end{proposition}
\begin{proof}
Since the statement is local, we may assume that $X$ is affine. 
We may write $\fra = \prod \fra _i ^{\lambda _i}$ 
for nonzero ideal sheaves $\fra _i$ and $\lambda _i \in \RR _{>0}$. 
We fix a positive integer $c$ which satisfies $c \ge \lambda _i$ for every $i$. 
Take general elements $f_{i1}, \ldots , f_{ic} \in \fra _i$, and 
let $D_{i1}, \ldots, D_{ic}$ be the corresponding effective Cartier divisors. 
If $\Delta = \frac{1}{c} \sum _{i,j} \lambda _i D_{ij}$, then
$(X, \Delta)$ is log canonical and $a_{E} (X, \Delta) = a_{E} (X, \fra)$ 
for every $f$-exceptional divisor $E$ (see  \cite[Lemma 4.2]{Nak2}). 

Let $\{ E_i \}$ be the set of all $f$-exceptional divisors. 
We write 
\[
f^* (K_X + \Delta) = K_Y + \widetilde{\Delta} +\sum _i (1- a_i) E_i, 
\]
where $\widetilde{\Delta}$ is the strict transform of $\Delta$ and $a_i = a_{E_i} (X, \Delta)$. 
Note that $1 - a_i \ge 0$ for every $i$, by assumption. 
We have 
\begin{align*}
0 = f^* (K_X + \Delta) \cdot E_1
= &(K_Y + E_1) \cdot E_1 + \widetilde{\Delta} \cdot E_1 - a_1 E_1 ^2 \\ &+ (1 - a_2) E_1 \cdot E_2 
+ (1 - a_3) E_1 \cdot E_3 + \sum _{i \not = 1, 2, 3} (1- a_i) E_1 \cdot E_i. 
\end{align*}
It is clear that we have
\[
(K_Y + E_1) \cdot E_1 \ge -2, \quad \widetilde{\Delta} \cdot E_1 \ge 0 , \quad
\sum _{i \not = 1, 2, 3} (1- a_i) E_1 \cdot E_i \ge 0,
\]
and the assumptions (1) and (2) give
\[
- a_1 E_1 ^2 \ge 2 a_1, \quad (1 - a_2) E_1 \cdot E_2 \ge 1 - a_2, \quad (1 - a_3) E_1 \cdot E_3 \ge 1 - a_3.
\]
By combining all these, we obtain the desired inequality $2 a_1 - a_2 - a_3 \le 0$. 
\end{proof}

\begin{proof}[Proof of Theorem \ref{dim2}]
Let $X$ be a klt surface, $x \in X$ a point and $I \subset \RR _{\ge 0}$ a finite set. 
The non-log-canonical case follows from Proposition \ref{prop_LC}, 
hence we only consider the log canonical case. 

Let $\fra$ be an $\RR$-ideal on $X$ with exponents in $I$ such that $(X,\fra)$ is log canonical around $x$.
Let $X_0\to X$ be the minimal resolution of $X$.
Suppose that $\mld _x (X, \fra)$ is not computed by any $(X_0 \to X)$-exceptional divisor. 
Then, there is a sequence of blow-ups
\[
X_n \to X_{n-1} \to \cdots \to X_1 \to X_0 \to X, 
\]
with the following properties:
\begin{enumerate}
\item For every $i$ with $0 \le i \le n-1$, the map $X_{i+1} \to X_i$ is the blow-up of $X_i$ at a point $p_i \in X_i$ 
with exceptional divisor $E_i \subset X_{i+1}$. 
\item $p_0$ maps to $x$ by the map $X_0 \to X$.   
\item $p_{i+1}$ maps to $p_i$ by the map $X_{i+1} \to X_i$ for every $i$ with $0 \le i \le n-2$ (equivalently, $p_{i+1} \in E_i$).
\item $a_{E_i}(X, \fra) > \mld _x (X, \fra)$ for $i$ with $0 \le i \le n-2$ and 
$a_{E_{n-1}} (X, \fra) = \mld _x (X, \fra)$. 
\end{enumerate}

The next lemma gives a bound for $k_{E_{n-1}}$ in terms of $n$. 
\begin{lemma}\label{lemma:n_to_k}
With the above notation, we have $k_{E_{n-1}} \le 2^{n-1}$. 
\end{lemma}
\begin{proof}
We first show that $\ord _{E_{n-1}} F \le 2^{n-1-i}$ for every prime divisor $F$ on $X_i$ 
which is exceptional over $X_0$ with $0 \le i \le n-1$. 
We argue by descending induction on $i$.
The case $i = n-1$ is trivial since each exceptional prime divisor over $X_0$ is smooth. 
If $i < n-1$, then the pull-back of $F$ to $X_{i+1}$ is either equal to the strict transform $F'$ of $F$ on $X_{i+1}$ or 
it is equal to $F' + E_{i}$. By induction, we conclude that 
$$\ord _{E_{n-1}} (F) \le \ord _{E_{n-1}} (F' + E_{i}) \le 2 \cdot 2^{n-2-i} = 2^{n-1-i}.$$ 

Note now that we have
\[
k_{E_{n-1}} = \ord _{E_{n-1}} (K_{X_n / X}) = \ord _{E_{n-1}} (K_{X_0 / X}) + \sum _{i = 1} ^n \ord _{E_{n-1}} (K_{X_i / X_{i-1}}). 
\]
On the other hand, since $X_0$ is the minimal resolution of $X$, we have $K_{X_0/X}\leq 0$, hence
$\ord _{E_{n-1}} (K_{X_0 / X}) \le 0$. Using the assertion at the beginning of the proof, we conclude
\begin{align*}
k_{E_{n-1}}\leq \sum _{i = 1} ^n \ord _{E_{n-1}} (K_{X_i / X_{i-1}}) = \sum _{i = 1} ^n \ord _{E_{n-1}} (E_{i-1}) \le 1 + \sum _{i = 1} ^{n-1} 2^{n-1-i} = 2^{n-1}, 
\end{align*}
which gives the desired inequality.
\end{proof}

Returning to the proof of Theorem~\ref{dim2}, it follows from
Lemma~\ref{lemma:n_to_k} that in order to conclude the proof of the theorem
it is enough to prove the following lemma, giving a bound on the number $n$ of blow-ups of $X_0$.
\end{proof}

\begin{lemma}\label{lemma:bound_n}
There exists a positive integer $\ell (I)$ depending on the finite set $I$ that satisfies the following condition: 
for every $\RR$-ideal $\fra$ on $X$ with exponents in $I$, 
if $(X, \fra)$ is log canonical and $\mld_x (X, \fra)$ is not computed by any $(X_0 \to X)$-exceptional divisor, 
then for every sequence of blow-ups satisfying the condition (1)-(4) above, we have
$n \le \ell(I)$. 
\end{lemma}

\begin{proof}
If $\mld _x (X, \fra) > 1$, then it is known that $X$ is smooth at $x$ 
(hence $X_0 = X$) and $n=1$ (see \cite[Theorem 4.5]{KollarMori} and its proof). 
From now on we suppose $\mld _x (X, \fra) \le 1$. 

We begin by proving the following assertion, which we will need 
in order to apply Proposition \ref{proposition:conv}:
\begin{equation}\label{eq:eff}
a_{E} (X, \fra) \le 1\quad\text{for every}\quad (X_n \to X)\text{-exceptional divisor}\,\,E. 
\end{equation}

Since $X_0 \to X$ is the minimal resolution, we have $K_{X_0/X}\leq 0$, hence
$$a_{E} (X, \fra) \le a_{E} (X) \le 1$$ 
 for every $(X_0 \to X)$-exceptional divisor $E$. 
Suppose that $j$ is the smallest index with $a_{E_j} (X, \fra) > 1$. 
We define an $\RR$-ideal $\fra _{j}$  on $X_{j}$ as follows:
if $\frb$ is the $\RR$-ideal on $X_j$ such that
$$\fra\cdot\cO_{X_j}=\frb\cdot\prod_E\cO_X(-E)^{\ord_E(\fra)},$$
where the product is over the  $(X_{j} \to X)$-exceptional divisors, then
$$\fra_j=\frb\cdot\prod_E\cO_X(-E)^{\ord_E(\fra)-k_E}$$
(note that this is well-defined since $\ord_E(\fra)-k_E=1-a_E(X,\fra)\geq 0$
for every such $E$). It follows from definition that
$$a_E (X, \fra) = a_E (X_{j}, \fra _{j})\quad\text{for every divisor}\,\,E\,\,\text{over}\,\,X.$$ 

Since $a_{E_j} (X_j, \fra_{j}) = a_{E_j} (X, \fra) > 1$, we have 
$
\mult _{p_{j}} \fra _{j} < 1. 
$
By \cite[Theorem 4.5]{KollarMori}, it follows that 
$
\mld _{p_{j}} (X_{j}, \fra _{j}) > 1. 
$
However, this contradicts
\[
\mld _{p_{j}} (X_{j}, \fra _{j}) 
\le a_{E_{n-1}} (X_{j}, \fra _{j}) = a_{E_{n-1}} (X, \fra) = \mld _x (X, \fra) \le 1. 
\]
This completes the proof of (\ref{eq:eff}).

Suppose now that $F_{0}, F_{1}, \ldots, F_{c}$ are $(X_n \to X_0)$-exceptional divisors 
that satisfy the following conditions: 
\begin{enumerate}
\item[($\alpha$)] $F_0 = E_{n-1}$ and $F_{i} \not = E_{n-1}$ for $1 \le i \le c$, and
\item[($\beta$)] $F_{i}$ meets $F_{i+1}$ for $0 \le i \le c-1$. 
\end{enumerate}
In this case we have the following sequence of inequalities:
\begin{equation}\label{claim2}
a_{E_{n-1}} = a_{F_0} < a_{F_1} < \cdots < a_{F_c}, 
\end{equation}
where we set $a_{F_i} = a _{F_i} (X, \Delta)$. 
In order to see this, note first that by the assumption on the sequence of blow-ups, we have
$a_{E_{n-1}} < a_{F}$ for every $(X_n \to X_0)$-exceptional divisor $F$ except for $F = E_{n-1}$. 
This gives the first inequality $a_{F_0} < a_{F_1}$. 
We next use the fact that $F ^2 \le -2$ for every $(X_n \to X_0)$-exceptional divisor $F$, except for $F = E_{n-1}$;
in particular, we have $F_1 ^2 \le -2$. 
It follows from Proposition \ref{proposition:conv} that
$$a_{F_1}\leq\frac{1}{2}(a_{F_0}+a_{F_2})<\frac{1}{2}(a_{F_1}+a_{F_2}).$$
Therefore $a_{F_1}<a_{F_2}$. We deduce in this way (\ref{claim2}) by repeatedly applying 
Proposition \ref{proposition:conv}.

By the discreteness of log discrepancies proved by Kawakita \cite{Kawakita2}, 
there exists a finite subset $U(I) \subset [0,1]$ depending only on $I$ satisfying the following condition:
\begin{itemize}
\item For every $\RR$-ideal $\fra$ with exponents in $I$ such that $(X, \fra)$ is log canonical, 
if $a_F (X, \fra) \in [0,1]$, then $a_F (X, \fra) \in U(I)$. 
\end{itemize}
Set $\ell_1 (I) := \# U(I)$. 
By the choice of $\ell_1 (I)$ and the bound (\ref{eq:eff}), if we can find a sequence $F_0, \ldots, F_c$ of exceptional divisors 
that satisfies the conditions ($\alpha$) and $(\beta)$ above, with $c \ge \ell_1 (I)$, we contradict the sequence of inequalities (\ref{claim2}).

The graph-theoretic Lemma~\ref{lem_graph} below thus implies
$n < \frac{1}{2}(3^{\ell_1 (I)} -1)$.
Indeed, we apply the lemma for the dual graph $\Gamma$ of $(X_n \to X_0)$-exceptional divisors 
(the vertices of this graph are given by these exceptional divisors and two vertices are connected by an edge if and only if 
the divisors intersect on $X_n$); note that $\Gamma$ 
has $n$ vertices and each vertex has degree at most three. This completes the proof of Lemma~\ref{lemma:bound_n}.
\end{proof}

\begin{lemma}\label{lem_graph}
Let $\ell$ be a positive integer and
$G$ be a connected graph of order $n \ge \frac{1}{2}(3^{\ell} -1)$. 
If every vertex of $G$ has degree $\leq 3$, 
then for any vertex $v$ of $G$, the graph $G$ contains a chain of length $\ell$ containing $v$ with degree 1.
\end{lemma}

\begin{proof}
We argue by induction on $\ell$, the case $\ell = 1$ being trivial. 
Consider the graph $G'$ obtained by removing the vertex $v$ and the edges containing $v$. 
Since $G$ is connected and ${\rm deg}(v)\leq 3$, 
the number of the connected components of $G'$ is at most three. 
Let $G''$ be a connected component of $G'$ of order at least 
$\frac{1}{3}\big( \frac{1}{2}(3^{\ell} -1) -1 \big) = \frac{1}{2}(3^{\ell -1} -1)$. 
Let $v'$ be a vertex in $G''$ which is connected to $v$ by an edge in $G$. 
By induction, $G''$ contains a chain of length $\ell -1$  containing $v'$ with degree $1$.
By adding $v$ to this chain, we obtain a chain in $G$ which contains $v$ with degree $1$.
\end{proof}

\section{A proof of the conjecture in the monomial case}

In this section we give a proof of Theorem~\ref{monomial_case}. More precisely, we prove the following result. 
A \emph{monomial} $\RR$-ideal on $\AAA^n$ is an $\RR$-ideal of the form $\fra=\prod_{j=1}^r\fra_j^{\lambda_j}$,
where each ideal $\fra_j$ is generated by monomials.

\begin{theorem}\label{monomial_case_version2}
Given a positive integer $n$ and a finite subset $I\subset \RR_{\geq 0}$, there is a positive integer $\ell$ ${\rm (}$depending on  $n$ and $I$${\rm )}$ such that
for every monomial $\RR$-ideal $\fra$ on $\AAA^n$ with exponents in $I$, there is a divisor $E$
that computes $\mld_0(X,\fra)$ and such that $k_E\leq \ell$. 
\end{theorem}

We will use the following result of Maclagan \cite[Theorem~1.1]{Maclagan}: given an infinite set ${\mathcal U}$ of monomial ideals in $k[x_1,\ldots,x_n]$,
then there are two ideals $I,J\in {\mathcal U}$ such that $I\subseteq J$. This implies that given any sequence $(I_m)_{m\geq 1}$ of monomial ideals in $k[x_1,\ldots,x_n]$,
there is a subsequence $(I_{j_m})_{m\geq 1}$ such that $I_{j_m}\supseteq I_{j_{m+1}}$ for all $m$. Indeed, note first that we may assume that each ideal $I$ is equal to $I_m$
for only finitely values of $m$, since otherwise our assertion is trivial. Since $k[x_1,\ldots,x_n]$ is Noetherian, we can find ideals in $\{I_m\mid m\geq 1\}$ that are maximal with respect to inclusion. 
By Maclagan's result, there are only finitely many such ideals  and by our assumption there are only finitely many $m$ with the property that $I_m$ is such a maximal ideal.
Therefore we can find $m_1\geq 1$ such that $I_{m_1}\supseteq I_m$ for infinitely many values of $m$. By repeating now the argument for the ideals $I_m$, with $m>m_1$ and
$I_m\subseteq I_{m_1}$, we obtain our assertion.

\begin{proof}[Proof of Theorem~\ref{monomial_case_version2}]
If the conclusion of the theorem fails, then there is a sequence $(\fra_m)_{m\geq 1}$ of monomial $\RR$-ideals on $\AAA^n$ and a sequence $(\ell_m)_{m\geq 1}$
with $\lim_{m\to\infty}\ell_m=\infty$ such that for every divisor $E$ over $\AAA^n$ that computes $\mld_0(\AAA^n,\fra_m)$, we have $k_E\geq\ell_m$. We will show that
this leads to a contradiction.

Let $\lambda_1,\ldots,\lambda_r$ be the elements of $I$. By assumption, we can write each $\fra_m$ as
$$\fra_m=\prod_{j=1}^r \fra_{m,j}^{\lambda_j},$$
where all $\fra_{m,j}$ are monomial ideals. As we have seen, it follows from Maclagan's result that after passing to a subsequence, we may assume
that $\fra_{m,1}\supseteq\fra_{m+1,1}$ for all $m\geq 1$. Repeating this for the $\fra_{m,2},\ldots,\fra_{m,r}$, it follows that after $r$ such steps, we may assume that
$\fra_{m,j}\supseteq\fra_{m+1,j}$ for all $m\geq 1$ and all $j$, with $1\leq j\leq r$.

In particular, it follows from Proposition~\ref{general_properties} that $({\rm mld}_0(\AAA^n,\fra_m))_{m\geq 1}$ is a weakly decreasing sequence.
On the other hand, a result of Kawakita  \cite[Theorem~1.2]{Kawakita2} says that the set of mld's on a fixed klt germ, for $\RR$-ideals with exponents in the finite set $I$,
is finite. We thus conclude that after passing one more time to a subsequence, we may assume that all $\mld_0(\AAA^n,\fra_m)$ take the same value
(possibly infinite).

 Let $E$ be a divisor over $\AAA^n$ that computes $\mld_0(\AAA^n,\fra_1)$. Given $m\geq 1$, since  $\fra_{1,j}\supseteq\fra_{m,j}$ for all $j$,
 it follows that
 $$\mld_0(\AAA^n,\fra_m)\leq k_E+1-\ord_E(\fra_m)\leq k_E+1-\ord_E(\fra_1)=\mld_0(\AAA^n,\fra_1).$$
 Therefore all the above inequalities are equalities. In particular, $E$ computes $\mld_0(\AAA^n,\fra_m)$ for all $m\geq 1$,
 a contradiction. This completes the proof of the theorem.
\end{proof}

\section{Connection with ACC}

Our goal in this section is to prove Theorem~\ref{thm_acc}, relating Conjecture~\ref{conj_main} to the
ACC conjecture.

\begin{proof}[Proof of Theorem~\ref{thm_acc}]
Suppose that we have a sequence $(\fra_i)_{i\geq 1}$ of $\RR$-ideals on $X$ with exponents in $J$ such that
each $(X,\fra_i)$ is log canonical around $x$ and with
$q_i=\mld_x(X,\fra_i)$,  the sequence $(q_i)_{i\geq 1}$ is strictly increasing. Since $q_i\leq\mld_x(X)$ for every $i$, it follows that 
$q:=\lim_{i\to\infty}q_i<\infty$. 

We may write $\fra_{i}=\prod_{j=1}^{r_i}\fra_{i,j}^{\lambda_{i,j}}$, where each $\fra_{i,j}$ is a nonzero ideal on $X$ with $x\in {\rm Cosupp}(\fra_{i,j})$ 
and each $\lambda_{i,j}$ is a nonzero element of $I$. Since $J$ is a DCC set, it follows that there is $\epsilon>0$ such that $\lambda_{i,j}\geq\epsilon$ for all $i$ and all $j$ with $1\leq j\leq r_i$. 
Let $F$ be a fixed divisor over $X$ with $c_X(F)=x$. For every $i\geq 1$, it follows from the fact that $(X,\fra_i)$ is log canonical around $x$ that
$$r_i\epsilon\leq \sum_{j=1}^{r_i}\lambda_{i,j}\leq \sum_{j=1}^{r_i}\lambda_{i,j}\cdot\ord_F(\fra_{i,j})\leq k_F+1.$$
First, this implies that the $r_i$ are bounded. Second, it implies that the $\lambda_{i,j}$ are bounded. 
After possibly passing to a subsequence, we may assume that $r_i=r$ for all $i\geq 1$. Furthermore, 
since $J$ is a DCC set, it follows that after possibly passing again to a subsequence, we may assume that
each sequence $(\lambda_{i,j})_{i\geq 1}$ is nondecreasing. Since we have seen that the sequence is bounded, it follows that
$\lambda_j:=\lim_{i\to\infty}\lambda_{i,j}<\infty$. 

We consider new $\RR$-ideals $\fra'_i=\prod_{j=1}^r\fra_{i,j}^{\lambda_j}$ for $i\geq 1$. We now show that $(X,\fra'_i)$ is log canonical around $x$ for $i\gg 0$. 
Note that also the set $J'=J\cup\{\lambda_1,\ldots,\lambda_r\}$ satisfies DCC, hence
$${\mathcal A}:=\{\lct_x(X,\frb)\mid  \frb\,\,\text{is}\,\,\text{an}\,\,\RR\text{-ideal on}\,\,X\,\,\text{with exponents in}\,\,J'\}$$
satisfies ACC (since we work on a fixed variety, this follows from  \cite[Theorem~4.2]{dFEM1}; for the general statement, see \cite[Theorem~1.1]{HMX}).
In particular, there is $M$ such that $\lct_x(X,\frb)\leq M$ for every $\RR$-ideal $\frb$ on $X$ with exponents in $J$. 
Note that we have
\begin{equation}\label{eq_limit}
\lim_{i\to\infty}(\lct_x(X,\fra'_i)-\lct_x(X,\fra_i))=0.
\end{equation}
Indeed, it follows from Proposition~\ref{general_properties} that for every $\delta>0$ and for every $i$ such that $\lambda_{i,j}\geq (1+\delta)^{-1}\lambda_j$ for all $j$, we have
$$\frac{1}{\delta+1}\cdot \lct_x(X,\fra_i) \leq \lct_x(X,\fra'_i)\leq \lct_x(X,\fra_i),$$
hence 
$$0\leq \lct_x(X,\fra_i)-\lct_x(X,\fra'_i)\leq \frac{\delta}{\delta+1}\cdot\lct_x(X,\fra_i)\leq\frac{M\delta}{\delta+1}.$$
This gives (\ref{eq_limit}). On the other hand, we have by assumption $\lct_x(X,\fra_i)\geq 1$ for all $i\geq 1$. Since the set ${\mathcal A}$ satisfies ACC,
we conclude from (\ref{eq_limit}) that $\lct_x(X,\fra'_i)\geq 1$ (hence $(X,\fra'_i)$ is log canonical around $x$) for all $i\gg 0$. After possibly ignoring the first few terms,
we may assume that $(X,\fra'_i)$ is log canonical around $x$ for every $i\geq 1$.

We now choose for every $i$ a divisor $E_i$ over $X$ which computes $\mld_x(X,\fra'_i)$. Since we assume that $X$ satisfies the assertion in 
Conjecture~\ref{conj_main} for $I=\{\lambda_1,\ldots,\lambda_r\}$,
we may and will assume that the set $\{k_{E_i}\mid i\geq 1\}$ is bounded above. Since we have
$$0\leq\mld_x(X,\fra'_i)=k_{E_i}+1-\sum_{j=1}^r\lambda_j \cdot\ord_{E_i}(\fra_{i,j}),$$
it follows that there is $B>0$ such that $\ord_{E_i}(\fra_{i,j})\leq B$ for all $i$ and $j$.
On the other hand, since $\lambda_{i,j}\leq\lambda_j$ for all $i$ and $j$, we have by Proposition~\ref{general_properties}
\begin{equation}\label{eq_bound_mld}
a_{E_i}(X,\fra'_i)=\mld_x(X,\fra'_i)\leq \mld_x(X,\fra_i)\leq a_{E_i}(X,\fra_i)=a_{E_i}(X,\fra'_i)+\sum_{j=1}^r(\lambda_j-\lambda_{i,j})\cdot\ord_{E_i}(\fra_{i,j}).
\end{equation}

Since the $\RR$-ideals $\fra'_i$ have exponents in the finite set $\{\lambda_1,\ldots,\lambda_r\}$, it follows from a result of Kawakita  \cite[Theorem~1.2]{Kawakita2}  that the set
$\{\mld_x(X,\fra'_i)\mid i\geq 1\}$ is finite. After possibly passing to a subsequence, we may thus assume that $\mld_x(X,\fra'_i)=A$ for every $i\geq 1$. 
We then conclude from (\ref{eq_bound_mld}) that 
\begin{equation}\label{eq_bound_mld2}
A\leq q_i\leq A+B\cdot \sum_{j=1}^r(\lambda_j-\lambda_{i,j}).
\end{equation}
Since $\lim_{i\to\infty}\lambda_{i,j}=\lambda_j$ for all $j$, it follows from (\ref{eq_bound_mld2}) by passing to limit that $q=A$. Using 
one more time (\ref{eq_bound_mld2}), we obtain $A\leq q_i\leq q=A$ for every $i$, hence the sequence $(q_i)_{i\geq 1}$ is constant, a contradiction.
\end{proof}

\section{Three equivalent conjectures}

We begin by stating the Generic Limit conjecture and the Ideal-adic Semicontinuity conjecture for minimal log discrepancies. 

Let $X$ be a klt variety over $k$ and $x\in X$ a closed point. 
Given a positive integer $r$ and $r$ sequences of nonzero coherent sheaves of ideals 
$(\fra^{(i)}_j)_{i\geq 1}$ on $X$, for $1\leq j\leq r$, 
the generic limit construction (see \S\ref{section:generic_limit}) gives an affine klt scheme $\widetilde{X}$, 
a closed point $\widetilde{x}\in\widetilde{X}$, 
and $r$ ideals $\widetilde{\fra}_1,\ldots,\widetilde{\fra}_r$ on $\widetilde{X}$. 
\begin{conjecture}[{Generic Limit conjecture, \cite[Conjecture 4.5]{Kawakita2}}]\label{conjecture:generic_limit}
For positive real numbers $\lambda_1, \ldots , \lambda _r$, there exists an infinite subset $S \subseteq \mathbb{Z} _{>0}$ 
such that the following hold:
\begin{itemize}
\item The ideals $\widetilde{\fra}_1,\ldots,\widetilde{\fra}_r$ are again
generic limits of the sequences of ideals $(\fra^{(i)}_1)_{i \in S},\ldots, (\fra^{(i)}_r)_{i \in S}$, and 

\item For every $i\in S$, we have
$$
\mld _{\widetilde{x}} (\widetilde{X}, \prod _{j = 1} ^r \widetilde{\fra}_j ^{\lambda _j}) = 
\mld _x (X, \prod _{j=1} ^r (\fra ^{(i)} _j)^{\lambda _j}).$$
\end{itemize}
\end{conjecture}

\begin{remark}\label{remark:oneineq}
%
Note that in the setting of the above conjecture, the inequality 
\[
\mld _{\widetilde{x}} (\widetilde{X}, \prod _{j = 1} ^r \widetilde{\fra}_j ^{\lambda _j}) \ge 
\mld _x (X, \prod _{j=1} ^r (\fra ^{(i)} _j)^{\lambda _j})
\]
can easily be guaranteed. Indeed, 
let $E$ be a divisor computing $\mld _{\widetilde{x}} (\widetilde{X}, \prod _{j = 1} ^r \widetilde{\fra}_j ^{\lambda _j})$. 
Take a positive integer $\ell$ such that $\ell\cdot \ord _E (\frm _{\widetilde{x}}) > \ord _E \widetilde{\fra}_j$ holds for each $j$. 
Then we have 
\[
\mld _{\widetilde{x}} (\widetilde{X}, \prod _{j = 1} ^r \widetilde{\fra}_j ^{\lambda _j}) = 
\mld _{\widetilde{x}} (\widetilde{X}, \prod _{j = 1} ^r (\widetilde{\fra}_j + \frm _{\widetilde{x}} ^{\ell} )^{\lambda _j}). 
\]
By Remark~\ref{remark_properties_generic_limit},
 there exists an infinite subset $S \subseteq \ZZ_{>0}$
such that the first condition in the conjecture holds and
\[
\mld _{\widetilde{x}} (\widetilde{X}, \prod _{j = 1} ^r (\widetilde{\fra}_j + \frm _{\widetilde{x}} ^{\ell} )^{\lambda _j})
=
\mld _x (X, \prod _{j=1} ^r (\fra ^{(i)} _j +  \frm _x ^{\ell} )^{\lambda _j})
\]
for every $i \in S$. 
Since 
\[
\mld _x (X, \prod _{j=1} ^r (\fra ^{(i)} _j +  \frm _x ^{\ell} )^{\lambda _j})
\ge
\mld _x (X, \prod _{j=1} ^r (\fra ^{(i)} _j)^{\lambda _j}), 
\]
we obtain the claimed inequality. 
\end{remark}

We now turn to the (uniform version of) Ideal-adic Semicontinuity conjecture for minimal log discrepancies.

\begin{conjecture}\label{ideal_adic}
Let $X$ be a klt variety and let $x\in X$ be a point defined by the ideal $\frm_x$. Given a finite set $I\subseteq \RR_{\geq 0}$, 
 there is a positive integer 
$s$ ${\rm (}$depending on $(X,x)$ and $I$${\rm )}$ such that the following holds: for every $\RR$-ideals
$\fra=\prod_{j=1}^r\fra_j^{\lambda_j}$ and $\frb=\prod_{j=1}^r\frb_j^{\lambda_j}$, with $\lambda_j\in I$ for all $j$, if $\fra_j+\frm_x^s=\frb_j+\frm_x^s$ for all $j$,
then $\mld_x(X,\fra)\geq 0$ if and only if $\mld_x(X,\frb)\geq 0$, and if this is the case\footnote{If this is not the case and $\dim(X)\geq 2$, then the two mlds are equal since they are both $-\infty$.}, then $\mld_x(X,\fra)=\mld_x(X,\frb)$.
\end{conjecture}

We now prove the result stated in the Introduction, saying that Conjectures~\ref{conj_main}, \ref{conjecture:generic_limit}, and \ref{ideal_adic}
are equivalent.

\begin{proof}[Proof of Theorem~\ref{thm_equivalence}]
We first show that Conjecture~\ref{conj_main} implies Conjecture~\ref{ideal_adic}. Suppose that Conjecture~\ref{conj_main} holds for $(X,x)$ and every finite set $I$.
Let $I$ be such a set. By assumption, there is a positive integer $\ell$ such that for every $\RR$-ideal
$\fra=\prod_{j=1}^r\fra_j^{\lambda_j}$ on $X$, with $\lambda_j\in I$ for all $j$,
there is a divisor
$E$ computing $\mld_x(X,\fra)$ with $k_E\leq\ell$. Let $\epsilon$ be the smallest nonzero element of $I$ and let $s$ be a positive integer that satisfies
$s>\frac{\ell+1}{\epsilon}$.
Suppose that $\fra$ and $\frb$ are as in Conjecture~\ref{ideal_adic}, with $\mld_x(X,\fra)\geq 0$.
We may and will assume that $\lambda_j>0$ for all $j$.

Let $\fra'_s=\prod_{j=1}^r(\fra_j+\frm_x^s)^{\lambda_j}$ and $\frb'_s=\prod_{j=1}^r(\frb_j+\frm_x^s)^{\lambda_j}$.
We assume that $\fra_j+\frm_x^s=\frb_j+\frm_x^s$ for all $j$, hence $\fra'_s=\frb'_s$.

Let $E$ be a divisor over $X$ which computes $\mld_x(X,\fra)$ such that $k_E\leq \ell$. 
In this case we have
$0\leq\mld_x(X,\fra)=k_E+1-\sum_{j=1}^r\lambda_j\cdot\ord_E(\fra_j)$, hence
$$\sum_{j=1}^r\lambda_j\cdot\ord_E(\fra_j)\leq\ell+1.$$
It follows from the choice of $\epsilon$ and $s$ that
$$s\cdot\ord_E(\frm_x)\geq s> \frac{\ell+1}{\lambda_j}\geq\ord_E(\fra_j)$$ for every $j$. Using Proposition~\ref{general_properties}, we obtain
$\mld_x(X,\fra)=\mld_x(X,\fra'_s)$. 
Since $\fra'_s=\frb'_s$, we have
$\mld_x(X,\frb'_s)=\mld_x(X,\fra'_s)=\mld_x(X,\fra)$ and since $\mld_x(X,\frb)\leq \mld_x(X,\frb'_s)$, we conclude that $\mld_x(X,\frb)\leq\mld_x(X,\fra)$. 

On the other hand, we have $\mld_x(X,\frb)\geq 0$. Indeed, if this is not the case, then by assumption we can find a divisor $F$ that computes
$\mld_x(X,\frb)$, with $k_F\leq \ell$. Therefore we have $k_F+1<\ord_F(\frb)$. We now use the fact that $\mld_x(X,\frb'_s)\geq 0$. First, this implies that
$\ord_F(\frb'_s)<\ord_F(\frb)$, and since we can write
$$\ord_F(\frb'_s)=\sum_{j=1}^r\lambda_j\cdot\min\{s\cdot\ord_F(\frm_x),\ord_F(\frb_j)\},$$
we conclude that there is $j$ such that $s\cdot\ord_F(\frm_x)<\ord_F(\frb_j)$. 
Second, it gives
$$\ell+1\geq k_F+1\geq \ord_F(\frb'_s)\geq \lambda_js\cdot \ord_F(\frm_x)\geq \epsilon s>\ell+1,$$
a contradiction. We thus conclude that $\mld_x(X,\frb)\geq 0$. 

 We can now run the same argument with the roles of $\fra$ and $\frb$ reversed, to conclude that $\mld_x(X,\frb)\geq\mld_x(X,\fra)$. Therefore $\mld_x(X,\frb)=\mld_x(X,\fra)$,
This completes the proof of the fact that Conjecture~\ref{conj_main} implies Conjecture~\ref{ideal_adic}. 

We now show that Conjecture~\ref{ideal_adic} implies Conjecture~\ref{conjecture:generic_limit}. 
Suppose that we are in the setting of Conjecture~\ref{conjecture:generic_limit} and let $s$ be the positive integer provided by
Conjecture~\ref{ideal_adic} for the set $I=\{\lambda_1,\ldots,\lambda_r\}$. 
By assumption, we have 
\[
\mld _x (X, \prod _{j=1} ^r (\fra ^{(i)} _j +  \frm _x ^{\ell} )^{\lambda _j})
=
\mld _x (X, \prod _{j=1} ^r (\fra ^{(i)} _j)^{\lambda _j})
\]
for every $\ell\geq s$ and every $i$. 
The argument in Remark~\ref{remark:oneineq} then implies that there is an infinite subset $S\subseteq\ZZ_{>0}$ that satisfies the first condition
in Conjecture~\ref{conjecture:generic_limit} and such that
\[
\mld _x (X, \prod _{j=1} ^r (\fra ^{(i)} _j +  \frm _x ^{\ell} )^{\lambda _j})
=
\mld _{\widetilde{x}} (\widetilde{X}, \prod _{j=1} ^r \widetilde{\fra}_j^{\lambda _j}), 
\]
for every $i\in S$. We thus have the conclusion in Conjecture~\ref{conjecture:generic_limit}.

Finally, we show that Conjecture~\ref{conjecture:generic_limit} implies Conjecture~\ref{conj_main}. 
Let $\lambda _1, \ldots , \lambda_r$ be the nonzero elements of the finite set $I$. 
If the assertion in Conjecture~\ref{conj_main} is not true, 
then for each positive integer $i$ there exist coherent ideals $\fra ^{(i)} _1, \ldots ,\fra ^{(i)} _r$
with the following property:
\begin{itemize}
\item $k_{E_i} \ge i$ holds for every divisor $E_i$ that computes $\mld _x (X, \prod _{j=1} ^r (\fra ^{(i)} _j)^{\lambda _j})$. 
\end{itemize}
We use the generic limit construction for $(\fra ^{(i)} _j) _{\ge 1}$ to obtain coherent ideal sheaves
$\widetilde{\fra}_1,\ldots,\widetilde{\fra}_r$ on $\widetilde{X}$. 
By applying successively\footnote{We need the first condition in Conjecture~\ref{conjecture:generic_limit}
in order to be able to apply Remark~\ref{remark:oneineq} to the resulting subsequences of ideals.}  Conjecture~\ref{conjecture:generic_limit} and Remark~\ref{remark:oneineq}, we get an infinite subset $S\subseteq\ZZ_{>0}$ such that

\begin{equation}\label{eq_thm_equivalence}
\mld _x (X, \prod _{j=1} ^r (\fra ^{(i)} _j +  \frm _x ^{\ell} )^{\lambda _j})
=
\mld _x (X, \prod _{j=1} ^r (\fra ^{(i)} _j)^{\lambda _j})
\end{equation}
for every $i\in S$. 
Let $\ell'$ be the $\ell$ provided by Theorem~\ref{thm_bound_ord_point}. 
It follows that for every $i\in S$, there is a divisor $E_i$ that computes
$\mld _x (X, \prod _{j=1} ^r (\fra ^{(i)} _j +  \frm _x ^{\ell} )^{\lambda _j})$
such that $\ord_{E_i} (\frm _x) \le \ell '$. 
The equality (\ref{eq_thm_equivalence}) implies that $E_i$  also computes  
$\mld _x (X, \prod _{j=1} ^r (\fra ^{(i)} _j)^{\lambda _j})$. 
Therefore we have 
\[
\ord _{E_i} (\fra ^{(i)} _j +  \frm _x ^{\ell}) = \ord _{E_i} (\fra ^{(i)} _j)
\]
for every $j$, hence
\[
\ord _{E_i} (\fra ^{(i)} _j) \le \ord _{E_i} (\frm _x^{\ell}) \le \ell \ell'. 
\]
If  $i\in S$ satisfies $i > \mld _x (X) - 1 + \ell \ell' \sum _{j=1} ^r \lambda _j$, 
then we have
\begin{align*}
a_{E_i} (X, \prod _{j=1} ^r (\fra ^{(i)} _j)^{\lambda _j}) 
&= k_{E_i} + 1 - \sum _{j=1} ^r \lambda _j \cdot \ord_{E_i} (\fra ^{(i)} _j) \\
&\ge i +1 - \ell \ell' \sum _{j=1} ^r \lambda _j > \mld _x (X). 
\end{align*}
This contradicts the fact that 
\[
a_{E_i} (X, \prod _{j=1} ^r (\fra ^{(i)} _j)^{\lambda _j}) 
= \mld _{x} (X, \prod _{j=1} ^r (\fra ^{(i)} _j)^{\lambda _j}) \le \mld _x (X). 
\]
We thus showed that Conjecture~\ref{conjecture:generic_limit} implies Conjecture~\ref{conj_main}, 
completing the proof of the theorem.
\end{proof}

\begin{remark}
By the equivalence of Conjectures~\ref{conj_main} and~\ref{conjecture:generic_limit}, 
Theorem~\ref{dim2} and Theorem~\ref{thm_acc} also follow from 
results of Kawakita, see  \cite[Proposition~4.8, Theorem~5.3]{Kawakita2}. 
\end{remark}

\providecommand{\bysame}{\leavevmode \hbox \o3em
{\hrulefill}\thinspace}

\end{document}